\documentclass[12pt]{article}

\usepackage{blindtext} 

\usepackage[sc]{mathpazo} 
\usepackage[T1]{fontenc} 
\usepackage{microtype} 

\usepackage[english]{babel} 

\usepackage[hmarginratio=1:1,top=32mm,columnsep=20pt]{geometry} 
\usepackage[hang, small,labelfont=bf,up,textfont=it,up]{caption} 
\usepackage{booktabs} 

\usepackage{lettrine} 

\usepackage{enumitem} 
\setlist[itemize]{noitemsep} 

\usepackage{abstract} 

\usepackage{titlesec} 
\renewcommand\thesection{\Roman{section}} 
\renewcommand\thesubsection{\roman{subsection}} 
\titleformat{\section}[block]{\large\scshape\centering}{\thesection.}{1em}{} 
\titleformat{\subsection}[block]{\large}{\thesubsection.}{1em}{} 

\usepackage{fancyhdr} 
\usepackage{titling} 

\usepackage{hyperref} 

\pretitle{\begin{center}\Huge\bfseries} 
\posttitle{\end{center}} 
\title{Analyticity of Exponential Dirichlet Series and Applications to the Approximate Controllability of Parabolic Equations}

\author{%
\textsc{M. Ouzahra}\\
\normalsize MASI Laboratory,  Sidi Mohamed Ben Abdellah University  \\
Graduate Normal School of Fes\\
\normalsize \href{mohamed.ouzahra@usmba.ac.ma}{mohamed.ouzahra@usmba.ac.ma} }
\date{}
\renewcommand{\maketitlehookd}{%
\begin{abstract}
In this paper, we investigate the analyticity of a class of exponential Dirichlet series. We then explicitly determine the coefficients of their power series decomposition and provide an estimate for the remainder. As an application, we study the approximate controllability property of linear parabolic equations with locally distributed or lumped controls by employing the moment method, which relies on the exponential Dirichlet series associated with the spectrum of the system's operator.

\end{abstract}

{\bf Keywords}: Moment method, Dirichlet series, approximate controllability, heat equation.
}

\newtheorem{theorem}{Theorem}[section]
\newtheorem{corollary}{Corollary}

\newtheorem{remark}{Remark}

\newtheorem{proof}{Proof}

\begin{document}

\maketitle

\section{Introduction}\label{sec1}

The analyticity of Dirichlet series, particularly in terms of the parameters they depend on, provides a critical tool for evaluating the behavior of systems in control theory. For example, the analytic continuation property of these series can clarify the approximate observability, controllability, and stability of evolution systems, facilitating the analysis and design of control strategies \cite{avdonin92, avdonin95, ball79, curtain78}. Several theoretical approaches have been developed to analyze the approximate controllability of PDEs. For parabolic equations with analytic coefficients, the Holmgren Uniqueness Principle provides a direct proof (see, e.g., \cite{hormander03, rosier07, zuazua09}). In the specific case of linear heat equations, approximate controllability follows from the property of null controllability, a relationship that has been extensively studied in \cite{zuazua00}. One widely used approach to establish null controllability for parabolic PDEs relies on Carleman estimates \cite{emanuilov}. The moment method, in particular, has been applied to prove the null controllability of the one-dimensional heat equation with localized controls \cite{russell}. Other significant contributions to the null controllability of PDEs include the works of D.L. Russell \cite{russell73}, G. Lebeau and L. Robbiano \cite{lebeau95}, A. Fursikov and O. Yu. Imanuvilov \cite{FI96}, and E. Fern\'andez-Cara \cite{FC97}.\\
 The moment method remains a cornerstone in the study of approximate controllability for parabolic systems, particularly in the design of actuators and sensors \cite{el jai 87, el jai 95, privat17}. A key aspect of the moment  method for approximate controllability, is its reliance on the analyticity of exponential Dirichlet series associated with the spectrum of the system's governing operator (\cite{curtain78}, Lemma 3.14, p. 62, see also \cite{el jai 87,el jai 95}).

In this paper, we revisit the properties of exponential Dirichlet series  by providing a detailed proof of the analyticity of exponential Dirichlet series under new conditions on their defining parameters. We explicitly determine the coefficients of their power series decomposition and derive an estimate for the remainder term. As an application, we employ this framework to study the approximate controllability of linear parabolic equations with locally distributed or lumped controls using the moment method.

The paper is organized as follows: In Section 2, we establish sufficient conditions for the analyticity of real Dirichlet series. In Section 3, we apply the moment method to prove the approximate controllability of the heat equation.
Finally, Section 4 provides a conclusion.

\section{Analyticity of real exponential Dirichlet series}

We have the following result:

\begin{theorem}\label{thm}
Let \( (\lambda_j)_{j \geq 1}  \) and \( (\alpha_j)_{j \geq 1} \) be two real sequences such that:

$\star $ The $\lambda_j $ are assumed to be positive, except for at most a finite number of them.

 $\star $ The series $\sum_{j \geq 1} \alpha_j$ is sommable:
\[
   \sum_{j = 1}^{+\infty} |\alpha_j| < +\infty. \hspace{3cm} ({\cal H}_0)
\]
Let us define the function
\[
\varphi(t) = \sum_{j=1}^{+\infty} \alpha_j e^{-\lambda_j t}, \quad \forall t \geq 0.
\]
Then \( \varphi \) is analytic in \( \mathbb{R}^{*+} \). More precisely, for any \( \tau > 0 \), we have
\[
\varphi(t) = \sum_{n=0}^{\infty} b_n (t - \tau)^n, \quad \forall t \in (0, 2 \tau),
\]
where
\[
b_n = \sum_{j=1}^{\infty} \alpha_j e^{-\lambda_j \tau} \frac{(-\lambda_j)^n}{n!}.
\]
Furthermore, for all \( t \in (0, 2 \tau) \), we have
\begin{equation}\label{rest}
\left| \varphi(t) - \sum_{j=0}^{n} b_j (t - \tau)^j \right| = O\left( \frac{|t-\tau|^n }{\tau^n \sqrt{n}} \right), \quad \mbox{as } n \to +\infty.
\end{equation}
\end{theorem}

\begin{proof}

Without loss of generality, we  can suppose that all the $\lambda_j$ are positive. Then,  it is straightforward to verify that $\varphi$ is well-defined on $\mathbb{R}^+$.\\
Let $\tau \in \mathbb{R}^{+*}.$ Then, by expanding the exponential function around  $\tau \in \mathbb{R},$ the function $\varphi(t)$ can be represented as a double series:
    $$
    \varphi(t) = \sum_{j=1}^\infty \alpha_j e^{-\lambda_j \tau} \sum_{n=0}^\infty \left( \frac{(-\lambda_j)^n}{n!} (t - \tau)^n \right), \; \forall t \in \mathbb{R}^+.
    $$
For all $n \geq 0,$ we define:
 $$
        a_n = \sum_{j=1}^\infty |\alpha_j| e^{-\lambda_j \tau} \frac{(\lambda_j)^n}{n!}.$$
     We can easily show that
     $$        e^{-\lambda_j \tau} (\lambda_j)^n \leq \left( \frac{n}{e \tau} \right)^n,\; \forall j \geq 1, \forall n \geq 0.
     $$
     Then, using the Stirling estimate, we derive the asymptotic estimate:
     $$a_n = O\left( \frac{1}{\tau^n \sqrt{n}} \right),$$
     where  $O$ is the Landau symbol related to the domination property.\\
      Hence, for all $ t \in (0, 2 \tau), $ we have:
$$
\sum_{n=0}^\infty \left( \sum_{j=1}^\infty \left| \alpha_j e^{-\lambda_j \tau} \frac{(-\lambda_j)^n}{n!} \right| \right) |t - \tau|^n < +\infty.
$$
Applying  the  Lebesgue term-by-term integration theorem  with respect to the discrete measure, we deduce that
        $$
        \varphi(t) = \sum_{n=0}^\infty \left( \sum_{j=1}^\infty \alpha_j e^{-\lambda_j \tau} \frac{(-\lambda_j)^n}{n!} \right) (t - \tau)^n.
        $$
        which   can be write, for all $t \in (0, 2\tau),$ as
        $$
        \varphi(t) = \sum_{n=0}^\infty b_n (t - \tau)^n.
        $$
        with $$b_n = \sum_{j=1}^{\infty} \alpha_j e^{-\lambda_j \tau} \frac{(-\lambda_j)^n}{n!}.$$
      In other words,  $\varphi$ can be expanded as a power series at any point $\tau>0$.
     We conclude that $\varphi(t) $ is analytic on $(0,+\infty)$.\\
     Moreover, using that $$|b_n|\le a_n = O\left( \frac{1}{\tau^n \sqrt{n}} \right),$$ we deduce that
$$
\begin{array}{ccc}
  |\varphi(t)-\displaystyle \sum_{j=0}^n b_j (t - \tau)^j|   & \le & \displaystyle\sum_{j\ge n+1}^\infty a_j |t - \tau|^j\\
 \\
   & \le & \frac{C_\tau }{\sqrt n} \displaystyle\sum_{j\ge n+1}^\infty \frac{ |t - \tau|^j}{\tau^j} \\
   \\
   & = & \displaystyle\frac{C_\tau }{\sqrt n} \frac{ |t - \tau|^{n+1}}{\tau^{n+1}},
\end{array}
$$
where $C_\tau $ is a positive constant that depends on $\tau.$\\
Then the estimate (\ref{rest}) follows.

\end{proof}

%
%

%

\begin{remark}
 \begin{itemize}

   \item Another alternative for proving the analyticity in the lemma is as follows:  Without loss of generality, we can assume that $ \sum_{j \geq 1} \lambda_j |\alpha_j| < +\infty $, since if this is not the case, we can simply replace $\varphi$ by its primitive.
Furthermore,  we can assume  that $ (\lambda_j)_{j \geq 1} \subset \mathbb{R}^{*+} $.
Then, using the theorem of term-by-term differentiation, we can show that $ \varphi $ is of class $ \mathcal{C}^\infty $ on $ \mathbb{R}^{*+} $,
 and that for all $ k \in \mathbb{N} $,
$
\varphi^{(k)}(t) = \sum_{j=1}^{+\infty} (-1)^k \alpha_j \lambda_j^k e^{-\lambda_j t},\; \forall t > 0, \forall k \in \mathbb{N}
$, from which it comes:\\
$
\left|\frac{t^k \varphi^{(k+1)}(t)}{k!}\right| \leq \frac{k^k e^{-k}}{k!} \sum_{j=1}^{+\infty} \lambda_j |\alpha_j|,\; \forall t > 0.
$
     Hence, the remainder term  of the Taylor expansion, given by $ R_k(t) = \int_0^t \frac{s^k \varphi^{(k+1)}(s)}{k!} ds,\; t > 0, \; k \in \mathbb{N} $ satisfies the following estimate  for any $ 0 \leq a < b $:
$
\sup_{a \leq t \leq b} |R_k(t)| = O\left(\frac{k^k e^{-k}}{k!}\right) \to 0$, as $k\to+\infty$. Hence the analyticity of $\varphi$ on $(0,+\infty) $ follows.\\

 \item For the analyticity of \(\varphi\), one can also extend \(\varphi\) to \(\mathbf{C}\) by setting \(z = t + i s\), show that this extension is holomorphic on the half-plane \({\cal I}m(z) > 0\), and then use the equivalence between holomorphy and analyticity on an open set.

 \end{itemize}
\end{remark}

In the following result, we provide a relaxed  version of $ ({\cal H}_0) $ that ensures the analyticity of the exponential Dirichlet series.

\begin{corollary}
Let $(\lambda_j)_{j \geq 1} \subset \mathbb{R}^{*+}$ and $(\alpha_j)_{j \geq 1} \subset \mathbb{R}$ be two sequences such that for some $k \geq 0$, we have
$$
 \sum_{j \geq 1} \frac{|\alpha_j|}{|\lambda_j|^k} < +\infty, \hspace{3cm} ({\cal H}_k)
$$
and consider the function
\begin{equation}
\varphi(t) = \sum_{j=1}^{+\infty} \alpha_j e^{-\lambda_j t}, \quad t \in \mathbb{R}^+.
\end{equation}
Then $\varphi$ is analytic on $\mathbb{R}^{*+}.$
\end{corollary}

\begin{proof}
We define the function
\begin{equation}
\phi(t) = \sum_{j=1}^{+\infty} \frac{\alpha_j}{\lambda_j^k} e^{-\lambda_j t}, \quad \forall t \geq 0.
\end{equation}
From Theorem \ref{thm}, we deduce that $\phi$ is analytic in $\mathbb{R}^{*+}.$ Consequently, its $k$-th derivative,
\begin{equation}
\phi^{(k)}(t) = \varphi(t),
\end{equation}
is also analytic on $\mathbb{R}^{*+}.$
\end{proof}

The next result provides an application of the analyticity of the exponential Dirichlet series.

\begin{corollary}\label{cor2}
Let \( (\lambda_j)_{j \geq 1} \subset \mathbb{R} \) be a strictly increasing sequence, and let \( (\alpha_j)_{j \geq 1} \subset \mathbb{R} \) be a sequence such that the assumption $({\cal H}_k)$ holds for some \( k \geq 0 \).\\
Assume that, for some \( T > 0 \), the following holds:
\begin{equation} \label{*}
\sum_{j=1}^{+\infty} \alpha_j e^{-\lambda_j t} = 0, \quad \forall t \in [0,T].
\end{equation}
Then \( \alpha_j = 0 \) for all \( j \geq 1 \).
\end{corollary}

\begin{proof}
Without loss of generality, assume that \( (\lambda_j)_{j \geq 1} \subset \mathbb{R}^{*+} \). If this is not the case, we can multiply the relation~(\ref{*}) by \( e^{-\lambda_1 t} \) and replace \( \lambda_j \) by \( \lambda_j - \lambda_1 \).\\
By the previous corollary, the function defined by
\[
\varphi(t) = \sum_{j=1}^{+\infty} \alpha_j  e^{-\lambda_j t}, \quad \forall t \geq 0
\]
is analytic on \( \mathbb{R}^{*+} \). Therefore, from (\ref{*}), we deduce that \( \varphi(t) = 0 \) for all \( t \geq 0 \).\\
Dividing by \( e^{\lambda_1 t} \) and letting \( t \to +\infty \), we conclude that \( \alpha_1 = 0 \). Repeating this argument iteratively for the remaining indices, we conclude, by induction, that \( \alpha_j = 0 \) for all \( j \geq 1 \).\\
This completes the proof.

\end{proof}

\begin{remark} Under the assumption $({\cal H}_0),$ we recover, by the last corollary, the result of Lemma 3.14 in \cite{curtain78}.

\end{remark}

\section{Application to approximate controllability  of the heat equation}

\subsection{Preliminary on the approximate controllability}

Let us consider the linear system:
\begin{equation}\label{sl}
\left\{
\begin{array}{ll}
  \dot{z}(t) =& Az(t) + Bu(t), \\
  z(0) = & z_0\in H,
\end{array}
\right.
\end{equation}

where, the state  space $H$ is a Hilbert endowed  with an inner product $\langle\cdot,\cdot\rangle$ with associate  norm $\|\cdot\|$, and the control space $U$ is also a Hilbert space with Hilbertian norm $|\cdot|.$
\\
The operator $A: D(A) \subset H \to H$ is the dynamic of the system, which generates a $C_0$-semigroup $S(t).$
The control operator $B: U \to H$  is  assumed to be linear and bounded.

For any $u\in {\cal U}_T:=L^2(0,T;U),$ the space of Bochner-integrable functions, the system (\ref{sl}) admits a unique mild solution given by the following variation of constant formula (see \cite{pazy}, Chap. 4):
$$
z(t)=S(t)z_0+\int_0^t S(t-s)Bu(s)ds,\; \forall t\in [0,T].
$$
The system (\ref{sl}) is said to be exactly controllable at time $T, $ if for all initial state $z_0\in H$ and target state $z_1\in H$, there exists  $u\in {\cal U}_T$ such that $z(T)=z_1.$  This is equivalent to the surjectivity of the following bounded (controllability)  operator:
 $$
G : {\cal U}_T  \mapsto H;\; Gu= \int_0^t S(t-s)Bu(s)ds.
$$

\begin{remark}
If the state space $H$ is infinite-dimensional and if either $B$ or $S(t)$ is compact, then the Bair's theorem implies that the system
(\ref{sl}) cannot be exactly controllable on $[0,T]$ (see \cite{trigg77}).
\end{remark}


The above remark prevents us from considering the exact controllability for PDEs of parabolic type and implies that approximate controllability is more suitable for such equations, including the heat equation:
 Approximate controllability focuses on finding square integrable controls $u $ that steer the solution arbitrarily close to a desired target state,
 rather than requiring exact attainment of the target.  More precisely, system (\ref{sl}) is said to be approximately controllable at time $ T $ if, for any $ z_0 \in H $, $ z_1 \in H $,
  and $ \varepsilon > 0 $, there exists a control $ u \in L^2(0, T;U) $,  such that the respective solution of system (\ref{sl}) is globally defined on $ [0, T] $ and satisfies:
$$
\|z(T)-z_1\| \le \varepsilon\cdot
$$

In other words, system (\ref{sl}) is approximately controllable if the set of reachable states is dense in $H$. This is equivalent to the density of the range $\Im(G)$ of $G$ in $H$, which is also equivalent to $\ker(G^*)=\{0\}.$

Moreover, it can be easily seen that the adjoint operator $G^*$ of $G : H \rightarrow L^2(0,T;U)$ is given by
$$G^* z=B^*S^*(T-\cdot)z,$$
i.e., $(G^* z)(t)=B^*S^*(T-t)z$.

It follows that

$$y\in \ker(G^*) \Leftrightarrow B^*S^*(T-t)y=0,\; \forall t\in [0,T]$$
$$
\Leftrightarrow B^*S^*(t)y=0,\; \forall t\in [0,T].
$$

We then obtain the following characterization of approximate controllability (see \cite{curtain78}).

\begin{theorem}\label{thmca}
The system (\ref{sl}) is approximately controllable if and only if for every $y \in H$ (equiv. $y\in \ker(B^*))$, we have:
\begin{equation}\label{obsw}
 B^*S^*(t)y, \; \forall t \in [0,T] \Rightarrow y = 0.
\end{equation}
\end{theorem}

\subsection{Approximate controllability of the heat equation}

  Let $ T > 0, $ and let $\Omega $ be a bounded domain with smooth boundary $\partial\Omega.$ Suppose $ \omega $ be an open subset of $\Omega$,
and let ${\bf 1}_\omega$ denote  the characteristic function of $\omega.$ Then, we consider the following controlled system:

\begin{equation}\label{h1}
\left\{
\begin{array}{ll}
z_t(x,t) = \Delta z(x,t) + {\bf 1}_\omega u(x,t), & (x,t) \in \Omega \times (0,T), \\
z(x,t)  = 0, & (x,t) \in \partial\Omega \times (0,T), \\
z(x,0) = z_0(x), & x \in \Omega.
\end{array}
\right.
\end{equation}
Recall that the operator $ A = \Delta $ with domain $ D(A) = H_0^1(\Omega)\cap H^2(\Omega)  $ generates a $C_0$-semigroup $ S(t) $ on $ H:=L^2(\Omega), $ and consider the Hilbertian basis $\varphi_j, \; j \geq 1$ of $ L^2(\Omega) $, formed by the eigenfunctions of $A$,
associated with the eigenvalues $\mu_j, \; j \geq 1 $. \\
 We assume that the eigenvalues \( \mu_j \) of \( A \) are simple. This  condition  is satisfied, for instance, when \( n = 1 \) and \( \Omega \) is an open interval. In the case  where \( n = 2 \) and  \( \Omega=(0,a)\times (0,b) \) is a rectangular domain,  the condition holds if and only if \( \frac{a}{b} \not\in \mathbb{Q}\). For eigenvalues with multiplicity, which corresponds to several control actions, the analysis can be carried out similarly, by involving a rank condition when dealing with lumped controls,  as in \cite{curtain78}.

Let us now analyze the property of approximate controllability of (\ref{h1}). We will distinguish two situations, namely $(x,t)-$dependent controls and time-dependent controls.\\

{\bf A) Distributed controls: } Suppose $ u(t) := u(\cdot,t) \in U := L^2(\Omega)$, and consider the self-adjoint operator defined by:
$$
Bz = {\bf 1}_\omega z, \, \forall z \in L^2(\Omega).
$$
We will prove that (\ref{h1}) is approximately controllable over any interval $
[0,T]$ and for any open subset $\omega$ of $\Omega$. To this end, we shall demonstrate  that condition (\ref{obsw}) is satisfied. \\
We have the following equivalence for all $y \in L^2(\Omega)$:
$$
 {\bf 1}_\omega S(t)y = 0, \; \forall t \in [0,T]  \iff
 \sum_{j \geq 1} e^{\mu_j t} \langle y, \varphi_j \rangle \langle z, {\bf 1}_\omega \varphi_j \rangle = 0, \; \forall t \in [0,T], \forall z \in L^2(\Omega).
$$
Taking $\lambda_j=-\mu_j$ and $\alpha_j=\langle y, \varphi_j \rangle \langle z, {\bf 1}_\omega \varphi_j \rangle, \; j\ge 1,$
we can see that assumption $({\cal H}_0)$ holds. Then, according to Corollary  \ref{cor2},  we deduce that system (\ref{h1}) is approximately controllable on $[0,T]$ if and only if for all $y \in L^2(\Omega)$, we have:
$$
 \langle z, \langle y, \varphi_j \rangle {\bf 1}_\omega \varphi_j \rangle = 0, \forall z \in L^2(\Omega), \; \forall j \geq 1 \Rightarrow \langle y, \varphi_j \rangle = 0, \; \forall j \geq 1.
$$
This is equivalent to the condition: ${\bf 1}_\omega \varphi_j \ne 0$ for all $j \geq 1$, which is fulfilled for any $\omega$ with non null measure.
We conclude that the approximate controllability holds for any time $T > 0$.

\vspace{0.25cm}

{\bf B) Lumped controls: } In this part, we study the controllability of the system (\ref{h1}) with controls $ u \in L^2(0,T;\mathbb{R}) $ that do not depend on the spatial variable $x$ (i.e., $ u(\cdot,t) = u(t) \in \mathbb{R}, $ for all $t \geq 0 $). Hence, we can take $ U = \mathbb{R}, $ and the control operator $ B \in \mathcal{L}(\mathbb{R}, L^2(\Omega)) $ is defined by $ B : \lambda \mapsto \lambda {\bf 1}_\omega. $ Then, we have
$$ B^*y = \langle y, {\bf 1}_\omega \rangle,\; \forall  y \in L^2(\Omega).
 $$
 Furthermore, we have the following equivalence for all $y \in L^2(\Omega)$:
$$
 \langle  S(t)y, {\bf 1}_\omega\rangle = 0, \; \forall t \in [0,T]  \iff  \sum_{j \geq 1} e^{\mu_j t} \langle y, \varphi_j \rangle \langle  {\bf 1}_\omega, \varphi_j \rangle = 0, \; \forall t \in [0,T].
$$
According to Corollary  \ref{cor2}, the system (\ref{h1}) is approximately controllable by lumped control $u(t), $ if and only if
$$ \int_\omega \varphi_j(x) \, dx \neq 0,\; \; \forall  j \geq 1.$$

$\bullet$ Let us discuss this result in the context of the one-dimensional case: $\omega=(a,b),\; 0\le a<b\le 1$ with $ \Omega:=(0,1)$.\\

{\bf Case 1}: $a\pm b \not\in \mathbf{Q}.$

In this case, we have $\varphi_j=\sqrt{2} \sin(j\pi x)$ and $\lambda_j=-(j\pi)^2$ for all $j\ge 1.$

It follows that the condition
$$ \int_\omega \varphi_j(x) \, dx \neq 0,\; \; \forall  j \geq 1$$
 is equivalent to $a\pm b \not\in \mathbf{Q}.$

Hence, in this case, the approximate controllability is equivalent to  $a\pm b \not\in \mathbf{Q}.$

{\bf Case 2}: $a+ b \in \mathbf{Q}$ or $a- b \in \mathbf{Q}.$

In this case, we do not have global controllability.

Let us look for a subspace in which we can achieve approximate controllability.

Let $I:=\{i\in \mathbf{N}^*: \int_\omega \varphi_i(x) \, dx = 0\}$. Here, we have $I\ne \emptyset.$

Consider the closed subspace $M:=\{y\in H: BS(t)y=0,\; t\ge 0\}$. Here, we have
$$M=\left\{y\in H:\; y=\sum_{i\in I} \langle y,\varphi_i \rangle \varphi_i  \right\}.$$
Then its orthogonal subspace $V=M^\perp$ is a Banach space which, like $M$, is invariant under $S(t)$.\\
Note that $M\subset \ker(B^*)$, and so  $V\supset \ker(B^*)^\perp=\overline{\Im(B)}$. Moreover, $M$ and $V$ are  invariant under $BB^*$. Furthermore, the condition (\ref{obsw}) holds on $V.$ We conclude that the initial system (\ref{h1}) induces on $V$ (i.e., the projection of (\ref{h1}) on $V$) a system that is approximately controllable.

As an example of such a situation, in the case $\omega=(0,\frac{1}{2})$, the system (\ref{h1}) is not approximately controllable on $H=L^2(0,1).$ However, it is approximately controllable on the subspace
$$V=\left\{y\in L^2(0,1): \; y=\sum_{j\not \equiv  0 [4] } \langle y,\varphi_{j} \rangle \varphi_{j} \right\}.$$

\section{Conclusion}\label{sec13}

The analyticity of a family of exponential Dirichlet series has been established and its relevance to the moment method for approximate controllability has been demonstrated. The moment method was then revisited in the context of approximate controllability for parabolic partial differential equations, with particular emphasis on the prototypical example of heat equations controlled via localized or distributed inputs.

\end{document}